\documentclass[12pt]{amsart}
\usepackage{amssymb,amsmath,amsfonts,latexsym}
\usepackage{bm}

\usepackage{array,graphics,color}
\newcommand{\newsection}[1]{\setcounter{equation}{0} \section{#1}}
\numberwithin{equation}{section}

\newtheorem{propn}{Proposition}[section]
\newtheorem{thm}[propn]{Theorem}
\newtheorem{lemma}[propn]{Lemma}

\newtheorem*{thm*}{Theorem}
\theoremstyle{definition}

\newcommand{\Hil}{\mathcal{H}}

\newcommand{\Z}{\mathbb{Z}_+}

\newcommand{\T}{\mathcal{T}}

\newcommand{\h}{\hat{T}}

\newcommand{\Comp}{\mathbb{C}}

 \newcommand{\D}{\mathbb{D}}

 \newcommand{\Dt}{D_{T}}
 \newcommand{\Dtp}{D_{\h_{p}}}
 \newcommand{\Dtq}{D_{\h_{q}}}

\newcommand{\clb}{\mathcal{B}}

\newcommand{\cld}{\mathcal{D}}
\newcommand{\cle}{\mathcal{E}}
\newcommand{\clf}{\mathcal{F}}

\newcommand{\clh}{\mathcal{H}}
\newcommand{\clk}{\mathcal{K}}
\newcommand{\cll}{\mathcal{L}}

\newcommand{\clp}{\mathcal{P}}
\newcommand{\clq}{\mathcal{Q}}

\newcommand{\cls}{\mathcal{S}}

\newcommand{\z}{\bm{z}}
\newcommand{\w}{\bm{w}}
\newcommand{\raro}{\rightarrow}
\newcommand{\NI}{\noindent}

\begin{document}

\title[Isometric Dilations and von Neumann inequality]{Isometric Dilations and von Neumann inequality for a class of tuples in the polydisc}

\author[Barik]{Sibaprasad Barik}
\address{Department of Mathematics, Indian Institute of Technology Bombay, Powai, Mumbai, 400076, India}
\email{sibaprasadbarik00@gmail.com}

\author[Das] {B. Krishna Das}
\address{Department of Mathematics, Indian Institute of Technology Bombay, Powai, Mumbai, 400076, India}
\email{dasb@math.iitb.ac.in, bata436@gmail.com}

\author[Haria]{Kalpesh J. Haria}
\address{School of Basic Sciences, Indian Institute of Technology Mandi, Mandi, 175005, Himachal Pradesh,  India}
\email{kalpesh@iitmandi.ac.in, hikalpesh.haria@gmail.com}

\author[Sarkar]{Jaydeb Sarkar}
\address{Indian Statistical Institute, Statistics and Mathematics Unit, 8th Mile, Mysore Road, Bangalore, 560059, India}
\email{jay@isibang.ac.in, jaydeb@gmail.com}


%\today

\begin{abstract}
The celebrated Sz.-Nagy and Foias and Ando theorems state that a
single contraction, or a pair of commuting contractions, acting on a
Hilbert space always possesses isometric dilation and subsequently
satisfies the von Neumann inequality for polynomials in
$\mathbb{C}[z]$ or $\mathbb{C}[z_1, z_2]$, respectively. However, in
general, neither the existence of isometric dilation nor the von
Neumann inequality holds for $n$-tuples, $n \geq 3$, of commuting
contractions. The goal of this paper is to provide a taste of
isometric dilations, von Neumann inequality and a refined
version of von Neumann inequality for a large class of $n$-tuples,
$n \geq 3$, of commuting contractions.
\end{abstract}

\maketitle

\newsection{Introduction}

In this paper we investigate isometric dilation, von Neumann
inequality and a refined version of von Neumann inequality, in terms
of algebraic variety in the polydisc $\D^n$, for a
large class of $n$-tuples, $n \geq 3$, of commuting contractions on
Hilbert spaces. The set of all ordered $n$-tuples of commuting
contractions on a Hilbert space $\clh$ will be denoted as
$\T^n(\clh)$, that is
\[
\T^n(\clh) = \{(T_1, \ldots, T_n): T_i \in \clb(\clh),  \|T_i\| \leq
1, T_i T_j = T_j T_i, 1 \leq i, j \leq n\},
\]
where $\clb(\clh)$ denotes the set of all bounded linear operators
on $\clh$. Here we are mostly interested in $n$-tuples, $n \geq 3$,
of commuting contractions as it is well-known that a contraction, or
a pair of commuting contractions, admits isometric dilation and
hence, satisfies the von Neumann inequality (see Sz.-Nagy and Foias
\cite{NF} and Ando \cite{An}). A refined version of von Neumann
inequality, in the sense of algebraic varieties, also follows from
the recent papers \cite{AM1}, \cite{DS} and \cite{DSS}. More
specifically, here we are concerned with the validity of the
following three statements for tuples in $\T^n(\clh)$.

\vspace{0.1in}

\NI \textsf{Statement 1 (On isometric dilations):} \textit{Let $T
\in \T^n(\clh)$. Then there exist a Hilbert space $\clk (\supseteq
\clh)$ and an $n$-tuple of commuting isometries $V \in \T^n(\clk)$
such that $T$ dilates to $V$}.

Now any $n$-tuple of commuting isometries $V \in \T^n(\clk)$ can be
extended to an $n$-tuple of commuting unitaries, that is, there
exist a Hilbert space $\cll$ containing $\clk$ and an $n$-tuple of
commuting unitary operators $U \in \T^n(\cll)$ which extends $V$
\cite{NF}. Therefore, the celebrated von Neumann inequality is an
immediate consequence of Statement 1 (cf. \cite{NF}):

\vspace{0.1in}

\NI\textsf{Statement 2 (On von Neumann inequality):} If $T \in
\T^n(\clh)$, then for all $p \in \mathbb{C}[z_1, \ldots, z_n]$, the
following holds:
\[
\|p(T)\|_{\clb(\clh)} \leq \sup_{\z \in \D^n} |p(\z)|.
\]
\NI Here $\z$ denotes the element $(z_1, \ldots, z_n)$ in
$\mathbb{C}^n$, $z_i \in \mathbb{C}$, and $\D^n = \{\z \in
\mathbb{C}^n: |z_i| < 1, i = 1, \ldots, n\}$.

The next natural geometric and algebraic question to consider, after
Agler and McCarthy \cite{AM1}, is the existence of varieties in
$\D^n$ in the von Neumann inequality:

\vspace{0.1in}

\NI\textsf{Statement 3 (On a refined von Neumann inequality):} Let
$T \in \T^n(\clh)$. Then there exists an algebraic variety $V$,
depending on $T$, in $\D^n$ (or in $\overline{\mathbb{D}^n}$) such
that for all $p \in \mathbb{C}[z_1, \ldots, z_n]$, the following
holds:
\[
\|p(T)\|_{\clb(\clh)} \leq \sup_{\z \in V} |p(\z)|.
\]

As we hinted earlier, Statement 2, and hence Statement 1, fails
spectacularly in the sense that the von Neumann inequality does not
hold in general for $n \geq 3$. This result is due to Varopoulos
\cite{V} and Crabb and Davie \cite{CD1}. On the other hand, by
presenting a list of elementary counterexamples, Parrott \cite{Par}
proved that triples of commuting contractions do not, in general,
possess commuting isometric dilations. We refer the reader
interested in deep subtleties of von Neumann inequality for
$n$-tuples of commuting contractions to Choi and Davidson \cite{CD},
Drury \cite{Dr}, Holbrook \cite{Ho1, Ho2}, Knese \cite{K1},
Kosi\'{n}ski \cite{LK} and Pisier \cite{Pi}.

We also point out here an important difference between the Sz.-Nagy
and Foias dilation \cite{NF} for contractions and Ando dilation
\cite{An} for pairs of commuting contractions. In the former case,
the dilating isometries are explicit in the sense of the classical
Wold and von Neumann decomposition \cite{NF}. In the latter case,
dilating pairs of commuting isometries are complicated and mostly
unclassified. This leads us to a reformulation of Statement 1:

\NI \textsf{Statement 1* (On explicit isometric dilations):}
\textit{Let $T \in \T^n(\clh)$. Then there exist a Hilbert space
$\clk (\supseteq \clh)$ and an $n$-tuple of explicit (or tractable)
commuting isometries $V \in \T^n(\clk)$ such that $T$ dilates to
$V$.}

We refer the reader to \cite{AM1}, \cite{DS} and \cite{DSS} for
classes of pairs of commuting contractions with explicit (or
tractable) dilating isometries.

The above discussion leads naturally to the question of determining
$n$-tuples of operators in $\T^n(\clh)$, $n \geq 3$, satisfying
Statements 1, 2,  3 and 1*. This research direction is
still mostly unexplored except for the work of Grinshpan,
Kaliuzhnyi-Verbovetskyi, Vinnikov and Woerdeman \cite{VV}. More
specifically, and elegantly, Grinshpan, Kaliuzhnyi-Verbovetskyi,
Vinnikov and Woerdeman \cite{VV} proved the validity of Statement 2
for a large class of $n$-tuples of commuting strict contractions, $n
\geq 3$. In other words, if an $n$-tuple, $n \geq 3$, of commuting
strict contractions $T$ obeys certain positivity condition, then the
open unit polydisc is a spectral set for $T$. This also yields,
following Arveson's notion of completely bounded maps (see
\cite{A1}, \cite{A2} and Corollary 4.9 in \cite{Pi}), existence of
unitary dilations for those $n$-tuples of commuting strict
contractions. The main stimulus for their work was provided by
scattering theory, Schur-Agler class of functions and de
Branges-Rovnyak models \cite{LDR} in several variables. This is also
the spirit behind results by Cotlar and Sadosky \cite{MCS}, Agler
and McCarthy \cite{AM2}, Eschmeier and Putinar \cite{EP} and many
more.

In this paper, we introduce a large class, namely $\T^n_{p,q}(\clh)$
(see Subsection \ref{sub-Tnpq}), of $n$-tuples, $n \geq 3$, of
commuting contractions and show that they dilate to $n$-tuples of
explicit commuting isometries. Therefore, Statement 1* and hence
Statement 1 holds for tuples in $\T^n_{p,q}(\clh)$. This also allows
us to prove the von Neumann inequality for tuples in
$\T^n_{p,q}(\clh)$ (that is, Statement 2 holds). In particular, in a
larger context (see the examples in Subsection 2.3), we prove that
the Grinshpan, Kaliuzhnyi-Verbovetskyi, Vinnikov and Woerdeman's
$n$-tuples of operators \cite{VV} admit explicit isometric dilations
and hence yield the von Neumann inequality. Our recipe even provides
sharper results with new proofs of the results of Grinshpan,
Kaliuzhnyi-Verbovetskyi, Vinnikov and Woerdeman. Here, however, our
treatment of dilations and von Neumann inequality is conceptually
different. Our von Neumann inequality is even stronger for finite
rank $n$-tuples of operators in the sense of algebraic varieties
(and so, Statement 3 holds). Furthermore, our technique offers some
geometric, analytic and algebraic structural insight into the
positivity assumptions of $n$-tuples of operators. Our methodology
is motivated by the Hilbert module approach to multivariable
operator theory (cf. \cite{Sar}).

The rest of the paper is organized as follows. Section 2 introduces
terminology used throughout this paper. This section also gives a
list of motivating and non-trivial examples of tuples of commuting
contractions. Section 3 establishes the existence, with explicit
constructions, of isometric dilations for a large class of finite
rank $n$-tuples of commuting contractions. Using the isometric
dilations, in Section 4, we obtain a refined version of von Neumann
inequality (in terms of an algebraic variety) for finite rank
$n$-tuples of commuting contractions. Finally, in Section 5 we
consider the more general problem of describing isometric dilations
for $n$-tuples of commuting contractions. Sections 3 and 5 are
independent of each other.

\newsection{Definitions and Examples}

This section is aimed at providing definitions, motivating examples
and a known dilation theorem on $n$-tuples of commuting
contractions. First, we introduce some standard notation that will
be used in this paper. We denote
\[
\mathbb{Z}_+^n = \{\bm{k} = (k_1, \ldots, k_n) : k_i \in
\mathbb{Z}_+, i = 1, \ldots, n\}.
\]
Also for each multi-index $\bm{k} \in \mathbb{Z}_+^n$, commuting
tuple $T = (T_1, \ldots, T_n)$ on a Hilbert space $\clh$, and $\z
\in \mathbb{C}^n$ we denote
\[
T^{\bm{k}} = T_1^{k_1} \cdots T_n^{k_n},
\]
and
\[
\z^{\bm{k}} = z_1^{k_1} \cdots z_n^{k_n}.
\]

We begin with the definition of isometric dilations for $n$-tuples
of commuting contractions.

\subsection{Dilations of commuting tuples}\label{ss-dilation}

Let $\clh$ and $\clk$ be Hilbert spaces, and let $T \in \T^n(\clh)$
and $V \in \T^n(\clk)$. Then $V$ is said to be an \textit{isometric
dilation} of $T$ if $V$ is an $n$-tuple of commuting isometries and
there exists an isometry $\Pi : \clh \raro \clk$ such that $\Pi
T_i^* = V_i^* \Pi$ for all $i = 1, \ldots, n$. We also say that $T$
\textit{dilates} to $V$.

\NI In this case, for $\bm{k} \in \Z^n$, we have
\[
\Pi T^{*\bm{k}} = V^{* \bm{k}} \Pi,
\]
and so
\[
\Pi T^{*\bm{k}} \Pi^* = V^{* \bm{k}} \Pi\Pi^*,
\]
since
\[
(\Pi \Pi^*) V^{* \bm{k}} (\Pi \Pi^*) = V^{* \bm{k}} (\Pi \Pi^*),
\]
or, equivalently $V^{*\bm{k}} \clq \subseteq \clq$, where
\[
\clq = \textit{ran}\ \Pi.
\]
This immediately yields the following: $(T_1, \ldots, T_n)$ on
$\clh$ and $(P_{\clq} V_1|_{\clq}, \ldots, P_{\clq} V_n|_{\clq})$ on
$\clq$ are unitarily equivalent under the isometric isomorphism $\Pi
: \clh \raro \clq$, and
\[
(P_{\clq} V|_{\clq})^{* \bm{k}} = V^{* \bm{k}}|_{\clq},
\]
for all $\bm{k} \in \Z^n$. Here $P_{\clq}$ is the orthogonal
projection of $\clk$ onto $\clq$. Therefore, the $n$-tuple $T$ has a
power dilation to the $n$-tuple of commuting isometries $V$, in the
classical sense of Sz.-Nagy and Foias and Halmos.

The following example of isometric dilation is typical: Let
$H^2(\D^n)$, the Hardy space over $\D^n$, be the space of all
analytic functions $f = \sum_{\bm{k} \in \Z^n} a_{\bm{k}}
\z^{\bm{k}}$ on $\D^n$ for which the norm
\[
\|f\|_{H^2(\D^n)} = (\sum_{\bm{k} \in \Z^n}
|a_{\bm{k}}|^2)^{\frac{1}{2}} < \infty.
\]
Let $(M_{z_1}, \ldots, M_{z_n})$ denote the $n$-tuple of
multiplication operators on $H^2(\D^n)$ defined by
\[
(M_{z_i} f)(\w) = w_i f(\w),
\]
for all $f \in H^2(\D^n)$, $\w \in \D^n$ and $i = 1, \ldots,
n$.
Then $(M_{z_1}, \ldots, M_{z_n})$ is an
$n$-tuple of commuting isometries $(M_{z_1}, \ldots, M_{z_n})$ on
the Hardy space $H^2(\D^n)$. Now for a joint
$(M_{z_1}^*, \ldots, M_{z_n}^*)$-invariant subspace $\clq$ of $H^2(\D^n)$,
consider
\[
T_j = P_{\clq} M_{z_j}|_{\clq},
\]
and
\[
\Pi = i,
\]
where $i : \clq \hookrightarrow H^2(\D^n)$ is the natural inclusion
map. Then
\[
\Pi T_j^* = M_{z_j}^* \Pi,
\]
for all $j=1, \ldots, n$. This implies that $(M_{z_1}, \ldots,
M_{z_n})$ on $H^2(\D^n)$ is an isometric dilation of $(T_1, \ldots,
T_n)$ on $\clq$.

\subsection{Hardy space and dilations}

We denote by $\mathbb{S}_n$ the Szeg\"{o} kernel on $\D^n$, that is,
\[\mathbb{S}_n(\z,\w)= \prod_{i=1}^n(1-z_i\bar{w_i})^{-1},
\]
for all $\z, \w \in \D^n$. Then $H^2(\D^n)$ is known to be a
reproducing kernel Hilbert space with kernel $\mathbb{S}_n$. If
$\cle$ is a Hilbert space, then $H^2_{\cle}(\D^n)$ denotes the
$\cle$-valued Hardy space over $\D^n$. Also as usual,
$H^2_{\cle}(\D^n)$ will be identified with the Hilbert space tensor
product $H^2(\D^n) \otimes \cle$ via the natural unitary map
$\z^{\bm{k}} \eta \mapsto \z^{\bm{k}}\otimes \eta$ for all $\bm{k}
\in \Z^n$ and $\eta \in \cle$. It is a well-known fact that
$H^2_{\cle}(\D^n)$ is a reproducing kernel Hilbert space on $\D^n$
corresponding to the $\clb(\cle)$-valued kernel function
\[
(\z, \w) \in \D^n \times \D^n \raro \mathbb{S}_n(\z, \w) I_{\cle}.
\]
The $n$-tuple of multiplication operators $(M_{z_1}, \ldots, M_{z_n})$
on $H^2_{\cle}(\D^n)$ defined analogously by
\[
(M_{z_i} f)(\w) = w_i f(\w),
\]
for all $f \in H^2_{\cle}(\D^n)$, $\w \in \D^n$ and $i = 1, \ldots,
n$. Let $H^\infty_{\clb(\cle)}(\D^n)$ denote the set of all bounded
$\clb(\cle)$-valued analytic functions on $\D^n$. The following is a
well-known fact (cf. page 655 in \cite{BLTT}): If $X \in
\clb(H^2_{\cle}(\D^n))$, then $X M_{z_i} = M_{z_i} X$ if and only if
$X = M_{\Theta}$ for some $\Theta \in H^\infty_{\clb(\cle)}(\D^n)$.
Now note that
\[
\mathbb{S}_n^{-1}(\z, \w) = \sum_{\bm{k} \in \{0,1\}^n}
(-1)^{|\bm{k}|} \z^{\bm{k}} \bar{\w}^{\bm{k}},
\]
where $|\bm{k}| = \sum_i k_i$, $\bm{k} \in \Z^n$. With this
motivation, for every $T \in \T^n(\clh)$ we set
\[
\mathbb{S}_n^{-1}(T, T^*) = \sum_{\bm{k} \in \{0,1\}^n}
(-1)^{|\bm{k}|} T^{\bm{k}} T^{*\bm{k}}.
\]
The set of all $T \in \T^n(\clh)$ with $\mathbb{S}_n^{-1}(T, T^*)
\geq 0$ will be denoted by $\mathbb{S}_n(\clh)$, that is
\[
\mathbb{S}_n(\clh) = \{ T \in \T^n(\clh) : \mathbb{S}_n^{-1}(T, T^*)
\geq 0\}.
\]
A tuple $T = (T_1, \ldots, T_n) \in \T^n(\clh)$ is said to be
\textit{pure} if $\|T^{*m}_i h\| \raro 0$ for all $h \in \clh$ and
$i = 1, \ldots, n$.

The following theorem on pure $n$-tuples in $\mathbb{S}_n(\clh)$ is
one of the most definite and significant results in multivariable
dilation theory (see \cite{CV1} and \cite{MV}).

\begin{thm}\label{th-dc dil}
Let $T \in \mathbb{S}_n(\clh)$ be a pure tuple. If
\[
\Dt = \mathbb{S}_n^{-1}(T, T^*)^{1/2},
\]
and
\[
\cld_T =\overline{\text{ran}}~ \mathbb{S}_n^{-1}(T, T^*),
\]
then $\Pi: \Hil\to H^2_{\cld_T}(\D^n)$ defined by
\[
(\Pi h)(\z) = \sum_{\bm{k} \in \Z^n} {\z}^{\bm{k}} D_T T^{*\bm{k}} h,
\]
for all $\z \in \D^n$ and $h \in \clh$, is an isometry and $\Pi
T_i^* = M_{z_i}^*\Pi$ for all $i = 1, \ldots, n$. In particular, $T$
on $\clh$ dilates to $(M_{z_1}, \ldots, M_{z_n})$ on
$H^2_{\cld_T}(\D^n)$.
\end{thm}

In the sequel, the isometry $\Pi$ defined in the above theorem will
be referred to as \textit{canonical isometry} corresponding to $T$.

\subsection{Commuting tuples in $\T^n_{p,q}(\clh)$}\label{sub-Tnpq}

We now introduce the central object of this paper.

\textsf{Let $\clh$ be a Hilbert space, and let $n \geq 3$ and $1\le
p<q\le n$ be fixed throughout the article.} Let $T \in \T^n(\clh)$.
For each $i \in \{1, \ldots, n\}$, we define
\[
\hat{T}_{i} = (T_1,\ldots,T_{i-1}, T_{i+1},T_{i+2}, \ldots, \dots,T_n) \in \T^{(n-1)}(\clh),
\]
the $(n-1)$-tuple obtained from $T$ by removing $T_i$. Define
\[
\T_{p,q}^n(\clh) = \{ T \in \T^n(\clh) : \h_p, \h_q \in
\mathbb{S}_{n-1}(\clh) \mbox{~and~} \h_p \mbox{~is pure}\}.
\]

For example, let $n = 3$, $p = 1$ and $q = 2$. Then $(T_1, T_2, T_3)
\in \T^3_{1,2}(\clh)$ if and only if:

\NI (i) $\|T_i\| \leq 1$ for all $i = 1, 2, 3$,

\NI (ii) $\h_1 = (T_2, T_3)$ is pure (that is, $\|T_i^{*m} h \|
\raro 0$ as $m \raro \infty$ for all $h \in \clh$ and $i = 2, 3$),

\NI (iii) $\mathbb{S}^{-1}_2(\h_1, \h_1^*) = I - T_2 T_2^* - T_3
T_3^* + T_2 T_3 T_2^* T_3^* \geq 0$, and

\NI (iv) $\mathbb{S}^{-1}_2(\h_2, \h_2^*) = I - T_1 T_1^* - T_3
T_3^* + T_1 T_3 T_1^* T_3^* \geq 0$.

Under the additional assumption that $\|T_i\| < 1$, $i = 1, \ldots,
n$, the above class of $n$-tuples of commuting contractions has been
studied, and denoted by $\clp_{p,q}^n(\clh)$, by Grinshpan,
Kaliuzhnyi-Verbovetskyi, Vinnikov and Woerdeman in \cite{VV}. It is
easy to see that $\|T_i\| < 1$, $i = 1, \ldots, n$, implies that
$(T_1, \ldots, T_n)$ is a pure tuple. More specifically, for every
$1 \leq p < q \leq n$, it is immediate that
\[
(M_{z_1}, \ldots, M_{z_n}) \in \T^n_{p,q}(H^2(\D^n)),
\]
but
\[
(M_{z_1}, \ldots, M_{z_n}) \notin \clp_{p,q}^n(H^2(\D^n)),
\]
and so
\[
\clp_{p,q}^n(H^2(\D^n)) \subsetneq \T^n_{p,q}(H^2(\D^n)).
\]
It should be noted, however, that $\clp_{p,q}^n(\clh)$ is a dense
subset of $\T^n_{p,q}(\clh)$ for any Hilbert space $\clh$. This
follows from the fact that if $T\in \T^n(\clh)$ and
$\mathbb{S}^{-1}_n(T,T^*)\ge 0$, then for any $0<r<1$

\[
 \mathbb{S}^{-1}_n(rT,rT^*)\ge 0,
\]
where
$rT=(rT_1,\dots,rT_n)$.

\subsection{Transfer functions}\label{sub-trans}

Our approach to isometric dilations and refined von Neumann
inequality will rely on the theory of transfer functions. Let
$\clh$, $\cle$ and $\cle_*$ be Hilbert spaces, and let $U: \cle
\oplus \Hil \rightarrow \cle_* \oplus \Hil$ be a unitary operator.
Assume that
\[
U=\begin{bmatrix}A &B\\ C& D\end{bmatrix} : \cle \oplus \Hil \raro
\cle_* \oplus \Hil.
\]
Then the \textit{transfer function} $\tau_{U}$ corresponding to $U$
is defined by
\[
\tau_{U}(z) = A + B z (I_{\Hil}- D z)^{-1}C,
\]
for all $z\in\D$. Since $\|D\| \leq 1$, and so $\|z D \| < 1$ for
all $z \in \D$, it follows that $\tau_U$ is a $\mathcal{B}(\cle,
\cle_*)$-valued analytic function on $\D$. Moreover, a standard and
well-known computation (cf. \cite{AMbook}) yields that
\begin{equation}\label{identity}
I - \tau_U(z)^* \tau_{U}(z) = (1 - |z|^2) C^* (I_{\Hil} - \bar{z}
D^*)^{-1}(I_{\Hil} - z D)^{-1}C,
\end{equation}
for all $z \in \D$. In particular, $\tau_U \in H^\infty_{\clb(\cle,
\cle_*)}(\D)$ and $\|M_{\tau_U}\| \leq 1$, that is, $\tau_U$ is a
contractive multiplier. We refer the reader to the monograph by
Agler and McCarthy \cite{AMbook} for more details.

\newsection{Dilations for finite rank tuples in $\T^n_{p,q}(\clh)$}
\label{3}

Let $T \in \T^n_{p,q}(\clh)$. We say that $T$ is of \textit{finite
rank} if
\[
\dim \cld_{\h_i} < \infty,
\]
for all $i = p, q$. In this section we find
explicit dilation for a finite rank $n$-tuple of commuting
contractions in $\T^n_{p,q}(\clh)$. Our (explicit) dilation result
seems to be especially more useful in studying refined von Neumann
inequality. Recall that for $T \in \T^n(\clh)$ and $i \in \{1,
\ldots, n\}$, $\h_i$ is defined as
\[
\hat{T}_{i} = (T_1,\ldots,T_{i-1}, T_{i+1},T_{i+2}, \ldots,T_n) \in \T^{(n-1)}(\clh).
\]
Let us also introduce the following notations, which will be
extensively used in the sequel. For $T \in \T^n(\clh)$, define
\begin{equation}\label{eq-Tn-2}
\hat{T}_{p,q} = (T_1,\dots,T_{p-1},
T_{p+1},\dots,T_{q-1},T_{q+1},\dots,T_n) \in \T^{(n-2)}(\clh),
\end{equation}
the $(n-2)$-tuple obtained from $T$ by deleting $T_p$ and $T_q$, and
\begin{equation}\label{eq-Tn-1}
\hat{T}_{pq} = (T_1,\dots,T_{p-1}, \underbrace{T_pT_q}_{p^{th}
~place},T_{p+1},\dots,T_{q-1},T_{q+1},\dots,T_n) \in
\T^{(n-1)}(\clh),
\end{equation}
the $(n-1)$-tuple obtained from $T$ by removing $T_q$ and replacing
$T_p$ by the product $T_pT_q$.

We begin with the following useful lemma on defect operators.

\begin{lemma}\label{lem-defect}
If $T \in \T_{p,q}^n(\clh)$, then
\[
D^2_{\h_{pq}} = \Dtp^2 + T_q\Dtq^2 T_q^*= T_p\Dtp^2 T_p^*+ \Dtq^2.
\]
\end{lemma}
\begin{proof}
Since by definition
\[
D^2_{\h_{p,q}} = \mathbb{S}_{n-2}^{-1}(\h_{p, q}, \h_{p,q}^*),
\]
it follows that
\[
D_{\h_p}^2 = D^2_{\h_{p,q}} - T_q D^2_{\h_{p,q}} T_q^*,
\]
and
\[
D_{\h_q}^2 = D_{\h_{p,q}}^2 - T_p D_{\h_{p,q}}^2 T_p^*.
\]
Then
\begin{align*}
D_{\h_{pq}}^2 & = \mathbb{S}_{n-1}^{-1}(\h_{pq}, \h_{pq}^*)
\\
&= D_{\h_{p,q}}^2 - T_pT_q D_{\h_{p,q}}^2 T_p^*T_q^*
\\
&= D_{\h_{p,q}}^2 - T_p D_{\h_{p,q}}^2 T_p^* + T_p(D_{\h_{p,q}}^2 -
T_q D_{\h_{p,q}}^2 T_q^*) T_p^*,
\end{align*}
that is
\begin{equation}\label{eq-defect}
D_{\h_{pq}}^2 = \Dtq^{2}+T_{p}\Dtp^2T_{p}^*,
\end{equation}
and similarly
\[
D_{\h_{pq}}^2 = \Dtp^{2}+T_{q}\Dtq^2T_{q}^*.
\]
This completes the proof of the lemma.
\end{proof}

Therefore, if $T \in \T_{p,q}^n(\clh)$, then it follows clearly from
the above lemma that the map
\[
U: \{\Dtp h\oplus \Dtq T_q^* h: h\in \Hil\}\to \{\Dtp T_p^* h\oplus
\Dtq h: h\in\Hil\}
\]
defined by
\[
\label{isometry1} U(\Dtp h, \Dtq T_q^* h)= (\Dtp T_p^* h, \Dtq h),
\]
for all $h\in\Hil$, is an isometry. In addition, if
\[
\dim \cld_{\h_i} < \infty,
\]
for all $i = p, q$, then $U$ extends to a unitary on
$\cld_{\h_p} \oplus \cld_{\h_q}$, which we denote again by $U$. This
implies the first part of the lemma below.

\begin{lemma}\label{unitary1}
If $T \in \T^n_{p,q}(\clh)$ is a finite rank tuple, then there
exists a unitary $U \in \clb(\cld_{\h_p} \oplus \cld_{\h_q})$ such
that
\[
U(\Dtp h, \Dtq T_q^* h)= (\Dtp T_p^* h, \Dtq h),
\]
for all $h \in \clh$. Moreover, if
\[
U = \begin{bmatrix} A& B\\ C& D\end{bmatrix} : \cld_{\h_p} \oplus
\cld_{\h_q} \raro \cld_{\h_p} \oplus \cld_{\h_q},
\]
then
\[
\Dtp T_p^*= A\Dtp + \sum_{i=0}^{\infty} B D^i C\Dtp T_q^{* i+1},
\]
where the series converges in the strong operator topology.
\end{lemma}
\begin{proof}
We only need to prove the second part. Let $h \in \clh$. Using
\[
U(\Dtp h, \Dtq T_q^* h)= (\Dtp T_p^* h, \Dtq h),
\]
we obtain
\[
\begin{bmatrix}A&B\\C&D\end{bmatrix} \begin{bmatrix} \Dtp h\\ \Dtq T_q^* h\end{bmatrix}=
\begin{bmatrix} \Dtp T_p^* h\\ \Dtq h\end{bmatrix}.
\]
Then
\[
\Dtp T_p^*h= A \Dtp h+ B \Dtq T_q^* h,
\]
and
\[
\Dtq h= C\Dtp h+ D \Dtq T_q^* h.
\]
Repeatedly resolving the former equation for $\Dtp T_p^* h$ in the
latter equation, we obtain
\[
\Dtp T_p^* h= A\Dtp h + \sum_{i=1}^m B D^{i} C\Dtp T_q^{*(i+1)}h + B
D^{m+1}\Dtq T_q^{*(m +2)}h,
\]
for all $h\in\Hil$ and $m \geq 1$. The proof now follows from the
fact that $T_q^{*m} h \raro 0$ as $m \raro \infty$ and $\|D\| \leq
1$.
\end{proof}

The proof of the second half of Lemma \ref{unitary1}, motivated by
\cite{DS}, will play an important role in what follows.

\begin{thm}\label{dilation1}
If $T \in \T^n_{p,q}(\clh)$ is a finite rank tuple, then there exist
an isometry $\Pi: \Hil \raro H^2_{\cld_{\h_p}}(\D^{n-1})$ and an
inner function $\varphi \in H^\infty_{\clb(\cld_{\h_p})}(\D)$ such
that
\[
\Pi T_i^* = \left \{\begin{array}{ll}
M_{z_i}^* \Pi \quad \quad \;\;\;\;\text{if~} 1\le i< p,\\
M_{\Phi_{p}}^* \Pi \quad \quad \;\;\;\text{if~} i = p,\\
M_{z_{i-1}}^* \Pi \quad \quad \;\text{if~} p < i\le n,
\end{array}
\right.
\]
where
\[
\Phi_{p}(\z) = \varphi(z_{q-1}),
\]
for all $\z \in \D^{n-1}$. In particular, $T$ on $\clh$ dilates to
the $n$-tuple of commuting isometries
\[
(M_{z_1},\dots,M_{z_{p-1}},M_{\Phi_{p}},M_{z_{p}},\dots,M_{z_{n-1}})
\]
on $H^2_{\cld_{\h_p}}(\D^{n-1})$.
\end{thm}
\begin{proof}
Since $\h_p \in \mathbb{S}_{n-1}(\clh)$ is a pure contraction, we
have
\[
\Pi T_i^* = \left \{\begin{array}{ll}
M_{z_i}^* \Pi \quad \quad \;\;\;\;\text{if~} 1\le i< p,\\
M_{z_{i-1}}^* \Pi \quad \quad \;\text{if~} p < i\le n,
\end{array}
\right.
\]
where $\Pi : \Hil \raro H^2_{\cld_{\h_p}}(\D^{n-1})$, defined by
\[
(\Pi h)(\z) = \sum_{\bm{k} \in \Z^{n-1}} \z^{\bm{k}} D_{\h_p} \h_p^{*
\bm{k}} h,
\]
for $\z \in \D^{n-1}$ and $h \in \clh$, is the canonical isometry
corresponding to $\h_p$ (see Theorem \ref{th-dc dil}). We prove that
\[
\Pi T_p^* = M_{\Phi_{p}} \Pi,
\]
for some one-variable (in $z_{q-1}$) inner function $\Phi_{p} \in
H^\infty_{\cld_{\h_p}}(\D^{n-1})$. To this end, consider $h\in\Hil$,
$\eta\in \cld_{\h_p}$ and $\bm{k} \in \Z^{n-1}$. Then
\[
\begin{split}
\langle \Pi T_p^*h, \z^{\bm{k}} \eta \rangle & = \langle \sum_{\bm{l}
\in \Z^{n-1}} \z^{\bm{l}} \Dtp \h_{p}^{* \bm{l}} T_p^* h, \z^{\bm{k}}
\eta \rangle
\\
& = \langle \Dtp T_p^* \h_{p}^{* \bm{k}} h, \eta \rangle.
\end{split}
\]
Next, consider the unitary
\[
U = \begin{bmatrix}A&B\\C&D\end{bmatrix} : \cld_{\h_p} \oplus
\cld_{\h_q} \raro \cld_{\h_p} \oplus \cld_{\h_q}
\]
as in Lemma \ref{unitary1}. Let
\[
\Phi_p(\z) = \tau_{U^*}(z_{q-1}),
\]
for all $\z \in \D^{n-1}$, where
\[
\tau_{U^*}(z) = A^*+zC^*(I_{\cld_{\h_q}} - zD^*)^{-1}B^*,
\]
for all $z\in\D$, is the transfer function corresponding to the
unitary map $U^*$. Since $\text{dim~} \cld_{\h_p} < \infty$, the
equality \eqref{identity} implies that $\tau_U$ is an inner
multiplier on $\D$. Also we compute
\[
\begin{split}
\langle M_{\Phi_p}^* \Pi h, \z^{\bm{k}} \eta \rangle & = \langle \Pi
h, \Phi_p (\z)(\z^{\bm{k}} \eta) \rangle
\\
&= \langle \sum_{\bm{l} \in \Z^{n-1}} \z^{\bm{l}} \Dtp {\h_p}^{*
\bm{l}} h, (A^*+ C^*\sum_{m=0}^{\infty} D^{* m}B^* z_{q-1}^{m+1})
\z^{\bm{k}} \eta \rangle
\\
& = \langle \Dtp \h_{p}^{*\bm{k}} h, A^*\eta\rangle+
\sum_{m=0}^{\infty} \langle \Dtp \h_{p}^{*\bm{k}}T_{q}^{* m+1} h,
C^*D^{*m}B^*\eta\rangle
\\
& = \langle A \Dtp \h_{p}^{*\bm{k}} h, \eta\rangle +
\sum_{m=0}^{\infty} \langle B D^m C \Dtp \h_{p}^{*\bm{k}}T_{q}^{*
m+1} h, \eta \rangle
\\
&=\langle (A \Dtp+ \sum_{m=0}^{\infty} B D^mC\Dtp T_q^{* m+1})
\h_{p}^{*\bm{k}} h, \eta\rangle,
\end{split}
\]
and so, by Lemma~\ref{unitary1}
\[
\langle M_{\Phi_p}^* \Pi h, \z^{\bm{k}} \eta \rangle = \langle \Dtp
T_p^* \h_{p}^{*\bm{k}} h, \eta \rangle.
\]
Thus
\[
\Pi T_p^* = M_{\Phi_p}^* \Pi.
\]
This completes the proof.
\end{proof}

The discussion in Subsection 2.1 gives another way to describe the
above dilation theorem: If $T \in \T^n_{p,q}(\clh)$, then $T$ and
\[
(P_{\clq} M_{z_1}|_{\clq},\ldots, P_{\clq} M_{z_{p-1}}|_{\clq} ,
P_{\clq} M_{\Phi_{p}}|_{\clq}, P_{\clq} M_{z_{p}}|_{\clq}, \ldots,
P_{\clq} M_{z_{n-1}}|_{\clq})
\]
on $\clq$ are jointly unitarily equivalent, where
\[
\clq = \text{ran~} \Pi \subseteq H^2_{\cld_{\h_p}}(\D^{n-1}),
\]
is a joint invariant subspace for
\[
(M_{z_1}^*, \ldots, M_{z_{p-1}}^*, M_{\Phi_{p}}^*, M_{z_{p}}^*,
\ldots, M_{z_{n-1}}^*).
\]

A natural question arises about the isometric dilation: What can be
said if the assumption of finite dimensionality in Theorem
\ref{dilation1} is removed? In the general case, the above ideas
allow one to prove that $\Phi_p$ is a contractive multiplier. We
proceed as follows: Let $\cld_{\h_p}$ or $\cld_{\h_q}$ is an infinite
dimensional Hilbert space. Let $\cld$ be an infinite dimensional
Hilbert space such that the isometry
\[
U: \{\Dtp h\oplus \Dtq T_q^* h: h\in \Hil\}\oplus \{0_{\cld}\}\to
\{\Dtp T_p^* h\oplus \Dtq h: h\in\Hil\}\oplus \{0_{\cld}\}
\]
defined by
\[
U(\Dtp h, \Dtq T_q^* h, 0_{\cld})= (\Dtp T_p^* h, \Dtq h, 0_{\cld}),
\]
for $h\in\Hil$, extends to a unitary, again denoted by $U$, on
\[
\cld_{\h_p} \oplus \cld_{\h_q} \oplus \cld.
\]
Then the same conclusion as in Lemma~\ref{unitary1} holds for the
unitary
\[
U = \begin{bmatrix}A&B\\C&D\end{bmatrix} \in \clb(\cld_{\h_p} \oplus
(\cld_{\h_q} \oplus \cld)).
\]
However, in this case, the transfer function $\tau_{U^*}$ is a
contractive analytic function. Then following the same line of
argument as in the proof of Theorem \ref{dilation1}, we have:

\begin{thm}
If $T \in \T^n_{p,q}(\clh)$, then there exist an isometry $\Pi: \Hil
\raro H^2_{\cld_{\h_p}}(\D^{n-1})$ and a contractive multiplier
$\varphi \in H^\infty_{\clb(\cld_{\h_p})}(\D)$ such that
\[
\Pi T_i^* = \left \{\begin{array}{ll}
M_{z_i}^* \Pi \quad \quad \;\;\;\;\text{if~} 1\le i< p,\\
M_{\Phi_{p}}^* \Pi \quad \quad \;\;\;\text{if~} i = p,\\
M_{z_{i-1}}^* \Pi \quad \quad \;\text{if~} p < i\le n,
\end{array}
\right.
\]
where $\Phi_p(\z)= \varphi(z_{q-1})$ for all $\z\in\D^{n-1}$.
\end{thm}

In Theorem \ref{without finite} we present a sharper version of the
above theorem: Every $n$-tuple in $\T^n_{p,q}(\clh)$ has an
isometric dilation. This will require a completely different method.
As is discussed in the following sections, the finite rank
assumption in Theorem \ref{dilation1} will turn out to be useful for
refined von Neumann inequality.

\section{von Neumann inequality for finite rank tuples in $\T^n_{p,q}(\clh)$}

In this section, we use the isometric dilations developed above to
prove the von Neumann inequality for finite rank tuples in
$\T^n_{p,q}(\clh)$.

First we recall the notion of completely non-unitary contractions
\cite{NF}. A contraction $T$ on $\clh$ is said to be
\textit{completely non-unitary} if there is no non-zero closed
reducing subspace $\cls \subseteq \clh$ for $T$ such that
$T|_{\cls}$ is a unitary operator. This notion has proved to be
useful in the following sense: If $T \in \T^1(\clh)$, then there
exists a unique decomposition $\clh = \clh_u \oplus \clh_c$ of
$\clh$ reducing $T$, such that $T|_{\clh_u}$ is unitary and
$T|_{\clh_c}$ is completely non-unitary. We therefore have the
\textit{canonical decomposition} of $T$ as:
\[
T = \begin{bmatrix} T|_{\clh_u} & 0 \\ 0 &
T|_{\clh_c}\end{bmatrix}.
\]
We need first the following result noted in ~\cite[Proposition
4.2]{DS}.

\begin{propn}\label{cd-U}
Let $U = \begin{bmatrix} A& B\\C & D\end{bmatrix}$ be a unitary
matrix on $\clh \oplus \clk$ and let $A =
\begin{bmatrix} A_u&0\\0&A_c\end{bmatrix} \in \clb(\clh_{u} \oplus \clh_{c})$ be
the canonical decomposition of $A$ into the unitary part $A_u$ on
$\clh_u$ and the completely non-unitary part
$A_c$ on $\clh_c$. Then $U'= \begin{bmatrix} A_c& B \\
C|_{\clh_c} & D
\end{bmatrix}$ is a unitary operator on $\clh_c \oplus \clk$ and
\[
\tau_U(z) =
\begin{bmatrix}
A_u & 0 \\ 0 & \tau_{U'}(z) \end{bmatrix} \in \clb(\clh_u \oplus
\clh_c) \quad \quad (z \in \mathbb{D}).
\]
\end{propn}

Now we turn to the distinguished varieties in $\D^2$ \cite{AM1}.
Recall that a non-empty set $V$ in $\mathbb{C}^2$ is a
\textit{distinguished variety} if there is a polynomial $p \in
\mathbb{C}[z_1, z_2]$ such that
\[V = \{(z_1, z_2) \in \mathbb{D}^2 : p(z_1, z_2) = 0\},\]and
$V$ exits the bidisc through the distinguished boundary, that is,
\[
\overline{V} \cap \partial \mathbb{D}^2 = \overline{V} \cap
(\partial \mathbb{D} \times \partial \mathbb{D}).
\]
Here $\partial \mathbb{D}^2$ and $\partial\mathbb{D}\times
\partial\mathbb{D}$ denote the boundary and the distinguished
boundary of the bidisc respectively, and $\overline{V}$ is the
closure of $V$ in $\overline{\mathbb{D}^2}$. We denote by $\partial
V$ the set $\overline{V} \cap \partial \mathbb{D}^2$, the boundary
of $V$ within the zero set of the polynomial $p$ and
$\overline{\mathbb{D}^2}$.

In the seminal paper \cite{AM1}, Agler and McCarthy characterized
distinguished varieties as follows: Let $V \subseteq \mathbb{C}^2$.
Then $V$ is a distinguished variety if and only if there exists a
rational matrix inner function $\Psi \in
H^\infty_{\clb(\mathbb{C}^m)}(\D)$, for some $m \geq 1$, such that
\[
V = \{(z_1, z_2) \in \mathbb{D}^2: \det (\Psi(z_1) - z_2
I_{\mathbb{C}^m}) = 0\}.
\]

In the following result, we
restrict ourselves to the case $\T^n_{1,2}(\clh)$, for simplicity.
The general case can be dealt with using a similar argument.

\begin{thm}\label{vN1}
If $T \in \T^n_{1,2}(\clh)$ is a finite rank operator, then there
exists an algebraic variety $V$ in $\overline{\D}^n$ such that for
all $p\in \Comp[z_1,\dots,z_n]$, the following holds:
\[
\|p(T)\|\le \sup_{\z\in V}|p(\z)|.
\]
If, in addition, $T_1$ is a pure contraction, then there exists a
distinguished variety $V'$ in $\D^2$ such that
\[
V = V' \times\D^{n-2} \subseteq \D^n.
\]
\end{thm}

\begin{proof}
Let $(M_{\Phi_1}, M_{z_1},\dots,M_{z_{n-1}})$ on
$H^2_{\cld_{\h_1}}(\D^{n-1})$ be the isometric dilation of $T$
provided by Theorem~\ref{dilation1}, where $\Phi_1\in
H^{\infty}_{\mathcal{B}(\cld_{\h_1})}(\D)$ is the inner multiplier
given by
\[
\Phi_1 = \tau_{U^*},
\]
and
\[
U^*=\begin{bmatrix}A^*& C^*\\B^*& D^*\end{bmatrix} \in
\mathcal{B}(\cld_{\h_1} \oplus \cld_{\h_2}),
\]
is the unitary as in Lemma \ref{unitary1}. Let
\[
A^*=\begin{bmatrix} A_u^*& 0\\0& A_c^*\end{bmatrix} \in
\mathcal{B}(\cld_u \oplus \cld_c),
\]
be the canonical decomposition of $A^*$ on
\[
\cld_{\h_1} = \cld_u \oplus \cld_c,
\]
into the unitary part $A_u^*$ on $\cld_u$ and the completely
non-unitary part $A_c^*$ on $\cld_c$. If we set
\[
U^*_c= \begin{bmatrix} A_c^* & C^* \\
B^*|_{\cld_c} & D^*
\end{bmatrix} \in \clb(\cld_c \oplus \cld_{\h_2}),
\]
then Proposition \ref{cd-U} implies that
\[
\Phi_1(z) = \begin{bmatrix}\Phi_u(z)&0\\0&\Phi_c(z)\end{bmatrix},
\]
where
\[
\Phi_u(z) \equiv A_u^*,
\]
and
\[
\Phi_c(z) = \tau_{U^*_c},
\]
for all $z\in\D$. Let
\[
V_u =\{(z,w)\in\bar{\D}^2: \det(zI_{\cld_u} - \Phi_u(w))=0\}\times
\D^{n-2}
\]
and
\[
V_c = \{(z,w)\in\D^2:  \det(zI_{\cld_c} - \Phi_c(w)) = 0\} \times
\D^{n-2}.
\]
Since $\Phi_c \in H^\infty_{\clb(\cld_c)}(\D)$ is a rational matrix
inner function (cf. page 138, \cite{AM1}), the characterization
result of distinguished varieties by Agler and McCarthy (Theorem
1.12 in \cite{AM1}) implies that
\[
\{(z,w)\in\D^2:  \det(zI_{\cld_c} - \Phi_c(w)) = 0\},
\]
is a distinguished variety in $\D^2$. Now let $p \in \Comp[z_1,
\dots, z_n]$. Since the discussion following Theorem~\ref{dilation1}
implies that $T$ on $\clh$ and
\[
(P_{\clq} M_{\Phi_1}|_{\clq}, P_{\clq} M_{z_1}|_{\clq}, \ldots,
P_{\clq} M_{z_{n-1}}|_{\clq}),
\]
on
\begin{equation}\label{eq-QranPi}
\clq = \text{ran~} \Pi \subseteq H^2_{\cld_{\h_1}}(\D^{n-1}),
\end{equation}
are unitarily equivalent, it follows that
\[
\|p(T)\|_{\clb(\clh)} = \|P_{\clq} p(M_{\Phi_1}, M_{z_1}, \ldots,
M_{z_{n-1}})|_{\clq}\|_{\clb(\clq)},
\]
and so
\[
\|p(T)\|_{\clb(\clh)} \leq \|p(M_{\Phi_1}, M_{z_1}, \ldots,
M_{z_{n-1}})\|_{\clb(H^2_{\cld_{\h_1}}(\D^{n-1}))}.
\]
But
\[
\|p(M_{\Phi_1}, M_{z_1}, \ldots,
M_{z_{n-1}})\|_{\clb(H^2_{\cld_{\h_1}}(\D^{n-1}))} =
\|M_{p(\Phi_1(z_{1}), z_1 I_{\cld_{\h_1}}, \ldots, z_{n-1}
I_{\cld_{\h_1}})}\|_{\clb(H^2_{\cld_{\h_1}}(\D^{n-1}))},
\]
and
\[
\|M_{p(\Phi_1(z_{1}), z_1 I_{\cld_{\h_1}}, \ldots, z_{n-1}
I_{\cld_{\h_1}})}\|_{\clb(H^2_{\cld_{\h_1}}(\D^{n-1}))} \leq
\|{p(\Phi_1 (z_1), z_1 I_{\cld_{\h_1}},
\ldots, z_{n-1}
I_{\cld_{\h_1}})}\|_{H^{\infty}_{\clb(\cld_{\h_1})}(\D^{n-1})}.
\]
Clearly, the right side is equal to
\[
\sup_{\theta_1,\dots,\theta_{n-1}} \|{p(\Phi_1 (e^{i \theta_1}),
e^{i \theta_1} I_{\cld_{\h_1}}, \ldots, e^{i \theta_{n-1}}
I_{\cld_{\h_1}})}\|_{\clb(\cld_{\h_1})}.
\]
Hence we have
\[
\|p(T)\|_{\clb(\clh)} \leq \sup_{\theta_1,\dots,\theta_{n-1}}
\|{p(\Phi_1 (e^{i \theta_1}), e^{i \theta_1} I_{\cld_{\h_1}},
\ldots, e^{i \theta_{n-1}} I_{\cld_{\h_1}})}\|_{\clb(\cld_{\h_1})}.
\]
Now for each $\theta_1,\dots,\theta_{n-1}$, the orthogonal
decomposition of
\[
\Phi_1(e^{i\theta_{1}}) = \Phi_u(e^{i\theta_{1}}) \oplus
\Phi_c(e^{i\theta_{1}}),
\]
on $\cld_{\h_{1}} = \cld_u \oplus \cld_c$ applied to
\[
p(\Phi_1 (e^{i \theta_1}), e^{i \theta_1} I_{\cld_{\h_1}}, \ldots,
e^{i \theta_{n-1}} I_{\cld_{\h_1}}) \in \clb(\cld_{\h_1}),
\]
shows that
\[
\begin{split}
\|{p(\Phi_1 (e^{i \theta_1}), e^{i \theta_1} I_{\cld_{\h_1}},
\ldots, e^{i \theta_{n-1}} I_{\cld_{\h_1}})}\|_{\clb(\cld_{\h_1})} &
= \max \{ \|{p(\Phi_u (e^{i \theta_1}), e^{i \theta_1} I_{\cld_u},
\ldots, e^{i \theta_{n-1}} I_{\cld_u})}\|_{\clb(\cld_u)},
\\
& \phantom{max.....} \|{p(\Phi_c (e^{i \theta_1}), e^{i \theta_1}
I_{\cld_c}, \ldots, e^{i \theta_{n-1}}
I_{\cld_c})}\|_{\clb(\cld_c)}\}.
\end{split}
\]
Note further that
\[
 \|{p(\Phi_s (e^{i \theta_1}), e^{i \theta_1} I_{\cld_s},
\ldots, e^{i \theta_{n-1}} I_{\cld_s})}\|_{\clb(\cld_s)} =
\mathop{\sup}_{\lambda \in \sigma(\Phi_s(e^{i\theta_{1}}))}
|p(\lambda, e^{i \theta_1}, \ldots, e^{i \theta_{n-1}})|,
\]
for all $s = u, c$, and hence
\[
\|{p(\Phi_1 (e^{i \theta_1}), e^{i \theta_1} I_{\cld_{\h_1}},
\ldots, e^{i \theta_{n-1}} I_{\cld_{\h_1}})}\|_{\clb(\cld_{\h_1})}
\leq \mathop{\sup}_{\lambda \in \sigma(\Phi_u(e^{i\theta_{1}})) \cup
\sigma(\Phi_c(e^{i\theta_{1}}))} |p(\lambda, e^{i \theta_1}, \ldots,
e^{i \theta_{n-1}})|.
\]
Consequently we have
\[
\begin{split}
\|p(T)\|_{\clb(\clh)} \leq & \sup_{\theta_1, \ldots, \theta_{n-1}}
\{|p(\lambda, e^{i \theta_1}, \ldots, e^{i \theta_{n-1}})| : \lambda
\in \sigma(\Phi_u(e^{i\theta_{1}})) \cup
\sigma(\Phi_c(e^{i\theta_{1}}))\}
\\
& = \sup_{\z \in \partial V_u \cup \partial V_c} |p(\z)|,
\end{split}
\]
and hence, by continuity
\[
\|p(T)\|_{\clb(\clh)}  \leq \|p\|_{V},
\]
where
\[
V = V_u \cup V_c.
\]
This proves the first part of the theorem. Assume now that $T_1$ is
a pure contraction. It is enough to prove that $V_u$ is an empty
set. Since $\Phi_u(z) \equiv A_u^*$, $z \in \D$, and $A_u$ is a
unitary on $\cld_u$, this is equivalent to proving that
\[
\cld_u = \{0\},
\]
which is further equivalent to the condition that $A^*$ is
completely non-unitary. First, we observe that (see
\eqref{eq-QranPi})
\[
\clq \subseteq H^2_{\cld_c}(\D^{n-1}).
\]
Indeed, let $g\in H^2_{\cld_u}(\D^{n-1})$, $m \in \Z$ and set $g_m =
M_{\Phi_u}^{*m} g$. Then
\[
g_m = A_u^{* m} g \in H^2_{\cld_u}(\D^{n-1}),
\]
and
\[
M_{\Phi_u}^{m} g_m = g.
\]
Now, if $f\in\clq$, then clearly
\[
\begin{split}
\langle g,f\rangle & = \langle M_{\Phi_u}^{m}g_m, f\rangle
\\
& = \langle g_m , T_1^{* m} f \rangle,
\end{split}
\]
and so
\[
\begin{split}
|\langle g,f\rangle | & \leq \|g_m\| \|T_1^{* m}f\|
\\
& = \|g\| \|T_1^{* m}f\|.
\end{split}
\]
Since $T_1$ is a pure contraction, it follows that
\[
\langle g,f\rangle = 0,
\]
and therefore $\clq\subseteq H^2_{\cld_c}(\D^{n-1})$. Finally, on
the one hand, by the property of the isometric dilation we have
\[
\bigvee_{\bm{k} \in \Z^{n-1}} M_{\z}^{\bm{k}}\clq =
H^2_{\cld_{\h_1}}(\D^{n-1}),
\]
and on the other hand $H^2_{\cld_c}(\D^{n-1}) \subseteq
H^2_{\cld_{\h_1}}(\D^{n-1})$ is a joint reducing subspace for $(M_{z_1},
\ldots, M_{z_{n-1}})$. Hence $H^2_{\cld_u}(\D^{n-1})=\{0\}$ and
therefore $\cld_u=\{0\}$. The theorem is proved.
\end{proof}

As we have pointed out before, the above von Neumann inequality for
tuples in $\T^n_{p,q}(\clh)$ is finer and conceptually different
(under the finite rank assumption) from the one obtained by
Grinshpan, Kaliuzhnyi-Verbovetskyi, Vinnikov and Woerdeman
\cite{VV}.

\newsection{Dilations for tuples in $\T^n_{p,q}(\clh)$}\label{sec-TPQ}

In this section we again investigate dilations for tuples in
$\T^n_{p,q}(\clh)$. We prove the validity of Statements 1* and 2 for
tuples in $\T^n_{p,q}(\clh)$ without any additional assumption on
the defect spaces $\cld_{\h_p}$ and $\cld_{\h_q}$.

Recall that (see Equations \eqref{eq-Tn-2} and \eqref{eq-Tn-1}) for
$T \in \T^n(\clh)$, we denote
\[
\hat{T}_{p,q} = (T_1,\dots,T_{p-1},
T_{p+1},\dots,T_{q-1},T_{q+1},\dots,T_n) \in \T^{(n-2)}(\clh),
\]
the $(n-2)$-tuple obtained from $T$ by deleting $T_p$ and $T_q$, and
\[
\hat{T}_{pq} = (T_1,\dots,T_{p-1}, \underbrace{T_pT_q}_{p^{th}
~place},T_{p+1},\dots,T_{q-1},T_{q+1},\dots,T_n) \in
\T^{(n-1)}(\clh),
\]
the $(n-1)$-tuple obtained from $T$ by removing $T_q$ and replacing
$T_p$ by the product $T_pT_q$. We begin with a simple but important
observation.

\begin{lemma}\label{product is szego}
If $T \in \T_{p,q}^n(\clh)$, then $\h_{pq}$ is a pure tuple and
$\h_{pq} \in \mathbb{S}_{n-1}(\clh)$.
\end{lemma}
\begin{proof}
Since $T_p T_q = T_q T_p$ and $T_q$ is a pure contraction, it follows that
$T_p T_q$ is a pure contraction, and hence $\h_{pq}$ is a pure tuple. On
the other hand, by (\ref{eq-defect}) we have
\[
D_{\h_{pq}}^2 = \Dtq^{2} + T_{p} \Dtp^2 T_{p}^*,
\]
and therefore,
\[
D_{\h_{pq}}^2 \geq 0,
\]
as $\h_p, \h_q \in \mathbb{S}_{n-1}(\clh)$. This completes the proof
of the lemma.
\end{proof}

Let $\cle_1$ and $\cle_2$ be two Hilbert spaces, and let
\[U =
\begin{bmatrix}
A&B\\C&0
\end{bmatrix},
\]
be a unitary operator on $\cle_1 \oplus \cle_2$. Then the
$\clb(\cle_1)$-valued transfer function $\tau_U$ on $\mathbb{D}$,
defined by (see Subsection \ref{sub-trans})
\[
\tau_U (z)= A + z B C,
\]
satisfies the equality (see \eqref{identity})
\[
I- \tau_U (z)^*\tau_U (z) = (1-|z|^2)C^*C,
\]
for all $z \in \D$. In particular, $\tau_U \in
H^\infty_{\clb(\cle_1)}(\D)$ is an inner function. Now if $1 \leq p
\leq n$ and
\[
\Phi(\z) = \tau_U(z_p),
\]
for all $\z \in \D^n$, then $\Phi \in H^\infty_{\clb(\cle_1)}(\D^n)$
is an inner polynomial in $z_p$ of degree at most 1. This point of
view will be used in what follows to develop the dilation theory for
tuples in $\T_{p,q}^n(\clh)$.

We now proceed to give an explicit description of isometric
dilations of tuples in $\T_{p,q}^n(\clh)$. Let $T \in
\T_{p,q}^n(\clh)$. Then, by the previous lemma, $\h_{pq} \in
\mathbb{S}_{n-1}(\clh) \cap \T^{(n-1)}(\clh)$ is a pure tuple. Let
\[
\cld_{\h_{pq}} = \overline{ran}\; \mathbb{S}_{n-1} (\h_{pq},
\h_{pq}^*),
\]
and let $\Pi_{pq} : \clh \raro H^2_{\cld_{\h_{pq}}}(\D^{n-1})$ be
the canonical isometry corresponding to $\h_{pq}$ (see Theorem
\ref{th-dc dil}). Then
\begin{equation}
\label{eq-Ri}
\Pi_{pq} R_i^* = (M_{z_i} \otimes I_{\cld_{\h_{pq}}})^* \Pi_{pq},
\end{equation}
for all $i = 1, \ldots, n-1$, where
\begin{equation*}
R_i = \left \{\begin{array}{ll}
T_i \quad \quad \quad \text{if~} 1\le i< q, i\neq p,\\
T_p T_q \quad \quad \text{if~} i = p,\\
T_{i+1} \quad \quad \text{if~} q \leq i\le n-1.
\end{array}
\right.
\end{equation*}
In other words, $(R_1, \ldots, R_{n-1}) = \hat{T}_{pq}$, that is
\[
(R_1, \ldots, R_{n-1}) = (T_1,\dots,T_{p-1},
\underbrace{T_pT_q}_{p^{th}
~place},T_{p+1},\dots,T_{q-1},T_{q+1},\dots,T_n),
\]
on $\clh$ dilates to
\[
(M_{z_1} \otimes I_{\cld_{\h_{pq}}}, \ldots, M_{z_{n-1}} \otimes
I_{\cld_{\h_{pq}}}),
\]
on $H^2_{\cld_{\h_{pq}}}(\D^{n-1})$ via the canonical isometry
$\Pi_{pq} : \clh \raro H^2_{\cld_{\h_{pq}}}(\D^{n-1})$. Now let
$\cle$ be a Hilbert space, and let $V : \cld_{\h_{pq}} \raro \cle$
be an isometry. Let
\begin{equation}\label{piv}
\Pi_{V, pq} = (I_{H^2(\D^{n-1})}\otimes V)\circ \Pi_{pq} \in
\clb(\clh, H^2_{\cle}(\D^{n-1})).
\end{equation}
Then $\Pi_{V,pq} : \clh \raro H^2_{\cle}(\D^{n-1})$ is an isometry
and
\begin{equation}\label{eq-piv}
\Pi_{V, pq} R_i^* = (M_{z_i}\otimes I_{\cle})^* \Pi_{V, pq},
\end{equation}
for all $i = 1, \ldots, n-1$. So $\h_{pq}$ on $\clh$ dilates to
$(M_{z_1} \otimes I_{\cle}, \ldots, M_{z_{n-1}} \otimes I_{\cle})$
on $H^2_{\cle}(\D^{n-1})$ via the isometry $\Pi_{V, pq}$.  Now we
are ready to prove the key lemma.

\begin{lemma}\label{dilating unitary}
Let $\clh$ and $\cle$ be Hilbert spaces, let $T \in
\T^n_{p,q}(\clh)$, and let $V$ and $\Pi_{V, pq}$ be as above. Let
$F_1$ and $F_2$ be bounded operators on $\Hil$, and let
$\clf_i=\overline{\text{ran}}\, F_i$, $i=1,2$. Let
\[
U_i = \begin{bmatrix}A_i&B_i\\C_i&0\end{bmatrix} :
\cle\oplus\clf_i\to\cle\oplus\clf_i,
\]
be a unitary operator, $i=1,2$. If
\[
U_1(V D_{\h_{pq}} h, F_1 T_p^*T_q^*h)=(V D_{\h_{pq}} T_p^*h, F_1h),
\]
and
\[
U_2(V D_{\h_{pq}} h, F_2 T_p^*T_q^*h)=(V D_{\h_{pq}} T_q^*h, F_2h),
\]
for all $h \in \clh$, then
\[
\Pi_{V, pq} T_p^*=M_{\Phi_1}^* \Pi_{V, pq},
\]
and
\[
\Pi_{V, pq} T_q^*=M_{\Phi_2}^* \Pi_{V, pq},
\]
where
\[
\Phi_i(\z)= A_i^* + z_p C_i^* B_i^* \quad \quad (\z\in \D^{n-1}),
\]
is the $\mathcal{B}(\cle)$-valued one variable transfer function of
$U_i^*$, $i = 1, 2$. In particular, ${\Phi_i} \in
H^\infty_{\clb(\cle)}(\D^{n-1})$, $i = 1, 2$, is an inner polynomial
in $z_p$ of degree at most 1.
\end{lemma}
\begin{proof}
Because of the symmetric roles of $T_p$ and $T_q$, we only prove
that $\Pi_{V, pq} T_p^* = M_{\Phi_1}^* \Pi_{V, pq}$. Let $h\in
\Hil$, $\bm{k} \in \Z^{n-1}$ and let $\eta\in\cle$. Using the
definition of $\Pi_{pq}$, we have
\[
\begin{split}
\langle \Pi_{V, pq} T_p^*h, {\z}^{\bm{k}} \otimes \eta \rangle & =
\langle (I_{H^2(\D^{n-1})} \otimes V) \Pi_{pq} T_p^*h, {\z}^{\bm{k}}
\otimes \eta \rangle
\\
& = \langle (I_{H^2(\D^{n-1})} \otimes V) \sum_{\bm{l} \in \Z^{n-1}}
{\z}^{\bm{l}} \otimes D_{\h_{pq}} \h_{pq}^{* \bm{l}} T_p^* h,
{\z}^{\bm{k}} \otimes \eta \rangle
\\
& = \langle V D_{\h_{pq}} \h_{pq}^{* \bm{k}} T_p^* h, \eta\rangle.
\end{split}
\]
Also since
\[
U_1(V D_{\h_{pq}} h, F_1 T_p^*T_q^*h)=(V D_{\h_{pq}} T_p^*h, F_1h),
\]
for $h\in\Hil$, we find that
\[
V D_{\h_{pq}} T_p^* = A_1 V D_{\h_{pq}} + B_1 F_1 T_p^*T_q^*,
\]
and
\[
F_1 = C_1 V D_{\h_{pq}}.
\]
Putting this together yields
\[
V D_{\h_{pq}} T_p^* = A_1 V D_{\h_{pq}} + B_1 C_1 V D_{\h_{pq}}
T_p^*T_q^*,
\]
and so
\[
\begin{split}
\langle M_{\Phi_1}^*\Pi_{V, pq} h, {\z}^{\bm{k}} \otimes \eta \rangle
& = \langle \Pi_{V, pq} h, M_{\Phi_1}({\z}^{\bm{k}} \otimes \eta)
\rangle
\\
& = \langle (I_{H^2(\D^{n-1})} \otimes V) \sum_{\bm{l} \in \Z^{n-1}}
{\z}^{\bm{l}} \otimes D_{\h_{pq}} \h_{pq}^{* \bm{l}} h, (A_1^* +
z_{p} C_1^* B_1^*) ({\z}^{\bm{k}} \otimes \eta) \rangle
\\
& = \langle A_1 V D_{\h_{pq}} \h^{* \bm{k}} h, \eta \rangle +
\langle B_1 C_1 V D_{\h_{pq}} \h_{pq}^{* \bm{k}} T_p^* T_q^* h, \eta
\rangle
\\
&= \langle V D_{\h_{pq}} \h_{pq}^{* \bm{k}} T_p^* h, \eta\rangle,
\end{split}
\]
and thus $\Pi_{V, pq} T_p^* = M_{\Phi_1}^* \Pi_{V, pq}$ as required.
The final claim follows easily from the paragraph following Lemma
\ref{product is szego}. This completes the proof of the lemma.
\end{proof}

Now we are ready to prove the main dilation result of this section.

\begin{thm}\label{without finite}
Let $\clh$ be a Hilbert space, and let $T =
(T_1, \ldots, T_n) \in \T^n_{p,q}(\clh)$. Then there exist a Hilbert
space $\cle$ and an isometry $\Pi: \Hil\to H^2_{\cle}(\D^{n-1})$
such that
\[
\Pi T_i^* = \left \{\begin{array}{ll}
M_{z_i}^* \Pi \quad \quad \quad \text{if~} 1\le i< q, i\neq p,\\
M_{\Phi_i}^* \Pi \quad \quad \; \;\;\text{if~} i = p, q,\\
M_{z_{i-1}}^* \Pi \quad \quad \;\text{if~} q < i\le n,
\end{array}
\right.
\]
where $\Phi_p$ and $\Phi_q$ in $H^\infty_{\clb(\cle)}(\D^{n-1})$ are
inner polynomials in $z_p$ of degree at most one and
\[
\Phi_p(\z) \Phi_q(\z) = \Phi_q(\z) \Phi_p (\z) = z_p I_{\cle},
\]
for all $\z \in \D^{n-1}$. In particular, $(T_1, \ldots, T_n) \in
\T^n_{p,q}(\clh)$ dilates to the isometric tuple
\[
(M_{z_1}, \dots, M_{z_{p-1}}, M_{\Phi_p}, M_{z_{p+1}},\dots,
M_{z_{q-1} }, M_{\Phi_q}, M_{z_{q}},\dots, M_{z_{n-1}} ),
\]
on $H^2_{\cle}(\D^{n-1})$ via the isometry $\Pi : \clh \raro
H^2_{\cle}(\D^{n-1})$.
\end{thm}

\begin{proof}
Using the identity in (\ref{eq-defect}), we have
\[
\begin{split}
D_{\h_{pq}}^2 & = \Dtq^{2}+T_{p}\Dtp^2T_{p}^* \\
& = \Dtp^{2} + T_{q}\Dtq^2 T_{q}^*,
\end{split}
\]
and then, for each $h \in \clh$, we have
\[
\begin{split}
\|D_{\h_{pq}} h\|^2 & = \|\Dtq T_q^*h\|^2+ \|\Dtp h\|^2
\\
& = \|\Dtq h\|^2+ \|\Dtp T_p^* h\|^2.
\end{split}
\]
This implies that the map
\[
U : \{ \Dtq T_{q}^*h,\Dtp h: h\in\Hil\} \raro \{\Dtq h, \Dtp
T_{p}^*h: h\in\Hil\},
\]
defined by
\[
(\Dtq T_{q}^*h, \Dtp h)\mapsto (\Dtq h, \Dtp T_{p}^*h),
\]
is a well-defined isometry. By adding, if necessary, an infinite
dimensional Hilbert space $\cld$, we extend $U$ to a unitary map,
again denoted by $U$, from $\cld \oplus \cld_{\h_q} \oplus
\cld_{\h_p}$ onto itself. Then, setting
\[
\cle= \cld \oplus \cld_{\h_q} \oplus \cld_{\h_p},
\]
we have a unitary map $U \in \clb(\cle)$ such that
\[
U(0_{\cld},\Dtq T_{q}^*h, \Dtp h )= (0_{\cld},\Dtq h, \Dtp T_{p}^*h
),
\]
for all $h \in \clh$. The equality
\[
\|D_{\h_{pq}} h\|^2 = \|\Dtq h\|^2+ \|\Dtp T_p^* h\|^2,
\]
again implies that the map $V : \cld_{\h_{pq}} \raro \cle$ defined
by
\[
V(D_{\h_{pq}} h)= (0_{\cld}, D_{\h_q} h, D_{\h_p} T_p^* h),
\]
for $h \in \clh$, is an isometry. Now by Lemma \ref{product is
szego}, it follows that $\h_{pq} \in \mathbb{S}_{n-1}(\clh)$ is a
pure tuple. Consider the canonical isometric map $\Pi_{pq} : \clh
\raro H^2_{\cld_{\h_{pq}}}(\D^{n-1})$ for $\h_{pq}$ such that
(\ref{eq-Ri}) holds. Then as in (\ref{piv}), set
\[
\Pi_{V, pq} = (I_{H^2(\D^{n-1})}\otimes V)\circ \Pi_{pq} \in
\clb(\clh, H^2_{\cle}(\D^{n-1})).
\]
Therefore, the isometry $\Pi_{V, pq}$ dilates $\h_{pq}$ on $\clh$ to
$(M_{z_1} \otimes I_{\cle}, \ldots, M_{z_{n-1}} \otimes I_{\cle})$
on $H^2_{\cle}(\D^{n-1})$. We now prove that
\[
\Pi_{V, pq} T_p^* = M_{\Phi_p}^* \Pi_{V,pq},
\]
and
\[
\Pi_{V, pq} T_q^* = M_{\Phi_q}^* \Pi_{V, pq},
\]
for some inner polynomials $\Phi_p, \Phi_q \in
H^\infty_{\clb(\cle)}(\D^{n-1})$ in $z_p$ variable and of degree at
most one and
\[
\Phi_p(\z) \Phi_q(\z) = \Phi_q(\z) \Phi_p (\z) = z_p I_{\cle},
\]
for all $\z \in \D^{n-1}$. To this end, let $\iota_p: \cld_{\h_p}
\hookrightarrow \cle$ and $\iota_q : \cld \oplus \cld_{\h_q}
\hookrightarrow \cle$ be the inclusion maps defined by
\[
\iota_p(h_p)=(0,0,h_p),
\]
and
\[
\iota_q(h, h_q)=(h, h_q, 0),
\]
for all $h_p \in \cld_{\h_p}, h_q \in \cld_{\h_q}$ and $h\in \cld$.
Let $P_{p}$ be the orthogonal projection of $\cle$ onto
$\cld_{\h_p}$. Since
\[
\left[
\begin{array}{cc}
P_p & \iota_{q} \\
\iota_{q}^* & 0 \\
\end{array}
\right] : \cle \oplus (\cld\oplus \cld_{\h_q}) \raro \cle \oplus
(\cld\oplus \cld_{\h_q}),
\]
is a unitary, it follows that
\[
U_1 = \left[
\begin{array}{cc}
U & 0 \\
0 & I \end{array}
\right]
\left[
\begin{array}{cc}
P_p & \iota_{q} \\
\iota_{q}^* & 0
\end{array}
\right],
\]
is a unitary operator on $\cle \oplus (\cld\oplus \cld_{\h_q})$.
Clearly
\[
U_1 =\left[\begin{array}{cc}
UP_{p}& U\iota_q\\
\iota_q^* & 0
\end{array}\right].
\]
We now prove that the unitary $U_1$ satisfies the condition of Lemma
\ref{dilating unitary}. Let $h \in \clh$. Then
\begin{align*}
U_{1}(V D_{\h_{pq}} h,0_{\cld}, \Dtq T_{p}^*T_{q}^*h) & =
U_{1}(0_{\cld},\Dtq h, \Dtp T_{p}^*h, 0_{\cld}, \Dtq T_{p}^*T_{q}^*
h)
\\
&=(U(0_{\cld},\Dtq T_{p}^*T_{q}^*h, \Dtp T_{p}^*h),0_{\cld}, \Dtq h)
\\
&=(0_{\cld}, \Dtq T_p^* h, \Dtp T_p^{* 2} h, 0_{\cld},\Dtq h)
\\
&=(V  D_{\h_{pq}} T_p^* h,0_{\cld}, \Dtq h).
\end{align*}
Similarly, if we consider the unitary
\[
U_{2} =\left[ \begin{array}{cc}
P_p^\perp & \iota_{p} \\
\iota_{p}^* & 0
\end{array}  \right] \left[ \begin{array}{cc}
U^* & 0 \\0 & I
\end{array} \right],
\]
on $\cle\oplus \cld_{\h_p}$, then, again using the fact that
\[
U_{2} =\left[ \begin{array}{cc}
P_p^\perp U^* & \iota_{p} \\
\iota_{p}^*U^* & 0 \\
\end{array}  \right],
\]
it follows that
\[
U_2 (V D_{\h_{pq}} h, D_{\h_p} T_p^* T_q^* h)=(V D_{\h_{pq}} T_q^*
h, \Dtp h),
\]
for all $h \in \Hil$. Therefore by Lemma~\ref{dilating unitary}, we
have $\Pi_{V, pq} T_p^* = M_{\Phi_p}^* \Pi_{V,pq}$ and $\Pi_{V, pq}
T_q^* = M_{\Phi_q}^* \Pi_{V, pq}$, where
\[
\Phi_p(\z)= (P_{p}+ z_p P_p^{\perp})U^*,
\]
and
\[
\Phi_q(\z)= U(P_p^{\perp}+ z_pP_p),
\]
for all $\z\in\D^{n-1}$, are the transfer functions corresponding to
the unitaries $U_1^*$ and $U_2^*$ respectively.
Also we have
\[
\Phi_p (\z) \Phi_q(\z) = \Phi_q(\z) \Phi_p(\z) = z_p I_{\cle},
\]
for all $\z\in\D^{n-1}$. This completes the proof of the theorem.
\end{proof}

Some remarks on the above dilation result are now in order.

\textsf{Remark 1:} For the base case $n = 3$, a closely related
result to Theorem \ref{without finite} was obtained in \cite{DSS}
as follows: Let $(T_1, T_2, T_3) \in \T^3(\clh)$, and let $T_3 = T_1
T_2$ be a pure contraction. Then $(T_1, T_2, T_3)$ on $\clh$ dilates
to $(M_{\Phi_1}, M_{\Phi_2}, M_z)$ on $H^2_{\cle}(\D)$ where $\cle$
is a Hilbert space, $\Phi_1, \Phi_2 \in H^\infty_{\clb(\cle)}(\D)$
are inner polynomials of degree $ \leq 1$, and
\[
\Phi_1(z) \Phi_2(z) = \Phi_2(z) \Phi_1(z) = z I_{\cle},
\]
for all $z \in \D$. Here $(M_{\Phi_1}, M_{\Phi_2})$ is a Berger,
Coburn and Lebow pair of commuting isometries \cite{BCL}. Our
approach to Theorem \ref{without finite} is partially motivated by
the above result. More specifically, in Theorem \ref{without finite},
the isometric pair $(M_{\Phi_p}, M_{\Phi_q})$ is a one
variable (in $z_p$) Berger, Coburn and Lebow pair of commuting
isometries on $H^2_{\cle}(\D^{n-1})$ in the following sense:
\[
\Phi_p(\z) \Phi_q(\z) = \Phi_q(\z) \Phi_p(\z) = z_p I_{\cle},
\]
for all $\z \in \D^{n-1}$.

\textsf{Remark 2:} Let $\cle$ be a Hilbert space, and let
$(M_{\varphi_1}, M_{\varphi_2})$ be a Berger, Coburn and Lebow pair
of commuting isometries on $H^2_{\cle}(\D)$, that is, ${\varphi_1}$
and ${\varphi_2}$ be two inner functions in
$H^\infty_{\clb(\cle)}(\D)$ and
\[
\varphi_1(z) \varphi_2(z) = \varphi_2(z) \varphi_1 (z) = z I_{\cle},
\]
for all $z \in \D$. For $1 \leq p < q \leq n$, define $\Phi_p(\z) =
\varphi_1(z_p)$ and $\Phi_q(\z) = \varphi_2(z_p)$, $\z \in \D^{n-1}$.
Then $\Phi_p$ and $\Phi_q$ in $H^\infty_{\clb(\cle)}(\D^{n-1})$ are
inner polynomials in $z_p$ of degree at most one, and
\[
\Phi_p(\z) \Phi_q(\z) = \Phi_q(\z) \Phi_p (\z) = z_p I_{\cle},
\]
for all $\z \in \D^{n-1}$. Let $\clq$ be a joint invariant subspace
for
\[
(M_{z_1}^*, \dots, M_{z_{p-1}}^*, M_{\Phi_p}^*, M_{z_{p+1}}^*,\dots,
M_{z_{q-1}}^*, M_{\Phi_q}^*, M_{z_{q}}^*,\dots, M_{z_{n-1}}^*),
\]
and let
\[
T_i = \left \{\begin{array}{ll}
P_{\clq} M_{z_i}|_{\clq} \quad \quad \quad \text{if~} 1\le i< q, i\neq p,\\
P_{\clq} M_{\Phi_i}|_{\clq} \quad \quad \; \;\;\text{if~} i = p, q,\\
P_{\clq} M_{z_{i-1}}|_{\clq} \quad \quad \;\text{if~} q < i\le n.
\end{array}
\right.
\]
It is then easy to see that $(T_1, \ldots, T_n) \in
\T^n_{p,q}(\clq)$. Therefore
\[
(M_{z_1}, \dots, M_{z_{p-1}}, M_{\Phi_p}, M_{z_{p+1}},\dots,
M_{z_{q-1}}, M_{\Phi_q}, M_{z_{q}},\dots, M_{z_{n-1}}),
\]
is the model $n$-tuple of isometries for $n$-tuples of commuting
contractions in $\T^n_{p,q}(\clh)$.

\textsf{Remark 3:} The previous remark gives a list of non-trivial
examples of $n$-tuples of operators in $\T^n_{p,q}(\clh)$. Also
observe that if $T \in \T^n(\clh)$ is doubly commuting, that is,
$T_i^* T_j = T_j T_i^*$ for all $1 \leq i < j \leq n$, then
\[
\mathbb{S}_{n-1}^{-1} (\h_p, \h_p^*) = \prod_{i \neq p} (I_{\clh} -
T_i T_i^*),
\]
for all $p \in \{1, \ldots, n\}$. Hence, if $T \in \T^n(\clh)$ is a
doubly commuting pure tuple, then $T \in \T^n_{p,q}(\clh)$ for any
$1 \leq p < q \leq n$. We refer to \cite{VV} for examples of
$n$-tuples of operators in $\T^n_{p,q}(\clh)$.

We conclude by recording the von Neumann inequality for tuples in
$\T^n_{p,q}(\clh)$. The proof follows easily, as pointed out earlier
(see the introduction), from the dilation result, Theorem
\ref{without finite}.

\begin{thm}
If $T \in \T^n_{p,q}(\clh)$, then for all $p \in \mathbb{C}[z_1,
\ldots, z_n]$, the following holds:
\[
\|p(T)\|_{\clb(\clh)} \leq \sup_{\z \in \D^n} |p(\z)|.
\]
\end{thm}

Note that the above von Neumann inequality generalizes the one
considered by Grinshpan, Kaliuzhnyi-Verbovetskyi, Vinnikov and
Woerdeman \cite{VV} to a large class of tuples in $\T^n(\clh)$ (see
Subsection \ref{sub-Tnpq}).

\vspace{0.1in} \noindent\textbf{Acknowledgement:} We are very
grateful to the referee for a careful reading of the manuscript, for
raising some interesting points, and for valuable comments and
corrections. The research of the second named author is supported by
DST-INSPIRE Faculty Fellowship No. DST/INSPIRE/04/2015/001094. The
third author's research work is supported by DST-INSPIRE Faculty
Fellowship No. DST/INSPIRE/04/2014/002624. The research of the
fourth named author is supported in part by the Mathematical
Research Impact Centric Support (MATRICS) grant, File No :
MTR/2017/000522, by the Science and Engineering Research Board
(SERB), Department of Science \& Technology (DST), Government of
India, and NBHM (National Board of Higher Mathematics, India)
Research Grant NBHM/R.P.64/2014.

\end{document}